\documentclass{amsart}
\usepackage{amsthm,amssymb,amsmath,amscd}
\usepackage{mathtools}
\usepackage{comment}
\usepackage{hyperref}

\newtheorem{theorem}{Theorem}[section]
\newtheorem{lemma}[theorem]{Lemma}
\newtheorem{proposition}[theorem]{Proposition}
\newtheorem{corollary}[theorem]{Corollary}
\newtheorem{remark}[theorem]{Remark}
\newtheorem*{definition}{Definition}

\DeclarePairedDelimiterX{\inp}[2]{\langle}{\rangle}{#1, #2}

\let\L\relax
\let\C\relax

\let\X\relax
\newcommand{\R}{\ensuremath{\mathbb{R}}}
\newcommand{\C}{\ensuremath{\mathbb{C}}}
\newcommand{\V}{\ensuremath{\mathbb{V}}}
\newcommand{\N}{\ensuremath{\mathcal{N}}}
\newcommand{\L}{\ensuremath{\mathcal{L}}}
\newcommand{\X}{\ensuremath{\mathfrak{X}}}

\DeclareMathOperator{\mg}{g}
\DeclareMathOperator{\K}{K}
\DeclareMathOperator{\Ad}{Ad}
\DeclareMathOperator{\ad}{ad}
\DeclareMathOperator{\tr}{tr}

\DeclareMathOperator{\id}{id}

\DeclareMathOperator{\diag}{diag}
\DeclareMathOperator{\spann}{span}

\renewcommand{\i}{\ensuremath{\mathrm{i}}}
\newcommand{\g}[1]{\ensuremath{\mathfrak{#1}}}

\begin{document}
\title[Geodesic completeness of Lorentzian simple Lie groups]{Geodesic completeness of some Lorentzian simple Lie groups}
\author[E.\ Ebrahimi]{E. Ebrahimi}
\author[S.\ M.\ B.\ Kashani]{S.M.B. Kashani}
\author[M.\ J.\ Vanaei]{M.J. Vanaei}

\address{Dept. of Pure Math., Faculty of Math. Sciences, Tarbiat Modares University, P.O. Box: 14115-134, Tehran, Iran}

\email{esmail.ebrahimi@modares.ac.ir}
\email{kashanim@modares.ac.ir}
\email{javad.vanaei@modares.ac.ir}


\begin{abstract}
In this paper we investigate geodesic completeness of left-invariant Lorentzian metrics on a simple Lie group $G$ when there exists a left-invariant Killing vector field $Z$ on $G$. Among other results, it is proved that if $Z$ is timelike, or $G$ is strongly causal and $Z$ is lightlike, then the metric is complete. We then consider the special complex Lie group $SL_2(\C)$ in more details and show that the existence of a lightlike vector field $Z$ on it, implies geodesic completeness. We also consider the existence of a spacelike vector field $Z$ on $SL_2(\C)$ and provide an equivalent condition for the metric to be complete. This illustrates the complexity of the situation when $Z$ is spacelike.
\end{abstract}

\maketitle

\section{Introduction}\label{section1}
Any invariant Riemannian metric on a homogeneous space $G/H$ is known to be geodesically complete (\cite{ONe83}, Remark~9.37). Invariant non-Riemannian metrics, however, require additional conditions in general to be geodesically complete. For example, if the homogeneous space is compact, then any invariant semi-Riemannian metric was proved to be complete \cite{Mar72}. 

Alekseevskii and Putko \cite{AlePut90} showed that geodesic completeness of invariant metrics on homogeneous spaces can be investigated through the case of $H=\{e\}$, that is, by studying Lie groups equipped with an invariant semi-Riemannian metric, which are referred to as semi-Riemannian Lie groups. 

Equations that determine geodesics in a semi-Riemannian Lie group are known as \textit{Euler equations}. This is due to the work of Arnold \cite{Arn13} who re-derived in terms of modern Lie theory the Euler equation which shows that the motion of a rigid body in the three-dimensional Euclidean space is described as the motion along geodesics in the group of rotations with an invariant metric. He then realized that the Euler equation can be extended to any arbitrary Lie group endowed with an invariant metric.

In case of semisimple (specially, simple) Lie groups, since the Killing form is non-degenerate, the geodesic equation translates into an Euler equation in the Lie algebra, allowing one to apply algebraic methods. Given a semi-Riemannian semisimple Lie group $(G,\mg)$, the Euler equation defines a homogeneous quadratic vector field $F_{\mg}$ on the Lie algebra $\g{g}$, called the \textit{Euler field}. Integral curves of $F_{\mg}$ are `reflections' of geodesics of $G$ in $\g{g}$. Indeed, integral curves of $F_{\mg}$ are complete if and only if geodesics of $\mg$ are complete. 

In this paper we study left-invariant Lorentzian metrics on simple Lie groups. We only consider non-compact groups as invariant metrics on compact groups are always complete. Since there is no distance associated to an indefinite metric, we use the notion of completeness in the sense of geodesic completeness. 

The dual space $\g{g}^*$ naturally admits a Poisson structure which can be induced on $\g{g}$ by identifying $\g{g}$ and $\g{g}^*$ via its non-degenerate Killing form (\cite{LauPicVan13},~Ch.~7). Then the Euler equation in $\g{g}$ represents a Hamiltonian system with the Hamiltonian function $\mg^*(x,x)$, where $\mg^*$ is the induced scalar product on $\g{g}^* \cong \g{g}$ from $\mg$. There is no systematic method in general to solve a Hamiltonian system or examine the completeness of its solutions. However, the more first integrals for a Hamiltonian system is available, the more can be said about the system. The Euler equation primitively has the first integrals $\mg^*(x,x)$ and $\tr(\ad_x^m)$, $m \in \mathbb{N}$. On $\g{sl}_2(\R)$, because of its low dimension, these first integrals are adequate to guarantee the integrability of the system. Using this fact, Bromberg and Medina \cite{BroMed08} fully characterised completeness of invariant Lorentzian metrics on $SL_2(\R)$ by proving that; a left-invariant Lorentzian metric $\mg$ on $SL_2(\R)$ is complete if and only if $F_{\mg}$ has no non-zero fixed point. Non-zero fix points of $F_{\mg}$, as an operator on $\g{g}$, are called \textit{idempotents}.

Finding new first integrals in dimensions greater than $3$, other than those mentioned above, is a major challenge. In Proposition~\ref{new-first-integral}, we show that if there exists a left-invariant Killing vector field $Z$, then one can obtain an additional first integral for the Euler equation. We then use this new first integral along with $\mg^*(x,x)$ and $\tr(\ad_x^m)$ to prove some properties of the adjoint operator $\ad_Z$, particularly, that $\ad_Z$ is not semisimple or nilpotent (Proposition~\ref{neither-semisimple-nilpotent}). Moreover, we see that if $Z$ is lightlike and $\mg$ is incomplete, then $\ad_Z$ is compact (Lemma~\ref{compact-element}), allowing us to obtain a causal curve lies in a compact subspace; the type of curves that `strongly causal' manifolds do not accept. As a result, we get our first main theorem as follows:
\begin{theorem}\label{MainTheorem1}
Let $(G,\mg)$ be a Lorentzian simple Lie group with a left-invariant Killing vector field $Z$ on it. Then, $\mg$ is complete in the following cases:  
\begin{itemize}
    \item[$(i)$]  if $Z$ is timelike, 
    \item[$(ii)$]  if $G$ is strongly causal and $Z$ is lightlike. 
\end{itemize}
\end{theorem}
Clearly,  Theorem~\ref{MainTheorem1} holds for globally hyperbolic simple Lie groups as they possess stronger condition than being strongly causal.  

The case that $Z$ is spacelike is much more challenging. To get some understanding of this case and also to compare the situation in dimension greater than $3$ with the special case of $SL_2(\R)$, we investigate left-invariant Lorentzian metrics on $SL_2(\C)$. To do so, we introduce a specific type of integral curves for the Euler field $F_{\mg}$ in $\g{g}$. We call an integral curve $u(t)$ of $F_{\mg}$ a \textit{generalized conical spiral} (GCS) if it satisfies $u(s) = r u(0)$ for some $r, s \in \R$ with $r>0$. The radial line generated by an idempotent is the graph of an special case of a GCS. Moreover, on $\g{sl}_2(\R)$ such radial lines are the only examples of GCS. So, the result of \cite{BroMed08} for $SL_2(\R)$ can be re-stated as: a left-invariant Lorentzian metric $\mg$ on $SL_2(\R)$ is complete if and only if $F_{\mg}$ has no GCS. Also, we will see that when there exists a left-invariant Killing vector field on $SL_2(\R)$, then the Lorentzian metric is complete (Proposition~\ref{killing-sl2r}).  Our next theorem states analogous result for $SL_2(\C)$.

\begin{theorem}\label{MainTheorem2}
Let $\mg$ be a left-invariant Lorentzian metric on $SL_2(\C)$. Suppose that there exists a left-invariant Killing vector field $Z$ on $SL_2(\C)$. Then, $\mg$ is complete if and only if $F_{\mg}$ has no GCS. Moreover, $F_{\mg}$ can have GCS only when $Z$ is spacelike. 
\end{theorem}
So, in particular, when $Z$ in Theorem~\ref{MainTheorem2} is timelike or lightlike, the metric $\mg$ is complete. It also follows that when there exists a left-invariant Killing vector field on $SL_2(\C)$, then $\mg$ is complete if and only if it is lightlike complete (Corollary~\ref{complete-lightlikecomplete}).

A different point of view was suggested in \cite{Tho14} to study and characterize completeness of invariant Lorentzian metrics on semisimple Lie groups. Let $\Lambda^*_{\mg}$ and $\N$ denote, respectively, the \textit{null cone} consisting of all null vectors in $\g{g}\cong \g{g}^*$ w.r.t $\mg^*$, and the \textit{nilpotent cone} which is the set of all non-zero nilpotent elements. In \cite{Tho14} it is shown that a left-invariant Lorentzian metric on $SL_2(\R)$ is incomplete if and only if $\Lambda^*_{\mg}$ and $\N$ are transversal. It is further claimed that this is true for any semisimple Lie group. We give a counterexample to this claim and prove the following theorem for $SL_2(\C)$ regarding this point of view on completeness. We then see that, in general, even completely determining $\Lambda^*_{\mg} \cap \N$ might be inconclusive for completeness of the metric.

\begin{theorem}\label{MainTheorem3}
Let $\mg$ be a left-invariant Lorentzian metric on $SL_2(\C)$ and the null cone $\Lambda^*_{\mg}$ and the nilpotent cone $\N$ are transversal. If there is no idempotent, then, $\Lambda^*_{\mg} \cap \N = T^2 \times \R$ where $T^2$ is a two-dimensional torus.
\end{theorem}

The notion of `limit sequence' or `limit curve' for a sequence of curves is widely used in general relativity (see e.g. \cite{ONe83}). In \cite{Yur92}, in order to prove that there exists an incomplete closed geodesic in any incomplete compact Lorentzian manifold, it was taken as a fact that a limit curve of any sequence of incomplete geodesics in a compact Lorentzian manifold is incomplete and closed. However, a counterexample was given in \cite{RomSan93} showing that the limit curve can be complete. The closedness of the limit curve was also argued in \cite{RomSan93} `unrigorously' without providing any example. In Subsection~\ref{limit-curve}, we show that there are examples satisfying the above assumptions and have non-closed limit curves.

The paper is organised as follows: In section~\ref{section2}, we provide preliminaries and set the notations. Section \ref{section3} is for the study of invariant Lorentzian metrics on simple Lie groups and the proof of Theorem~\ref{MainTheorem1}. We consider invariant Lorentzian metrics on $SL_2(\C)$, and prove Theorems~\ref{MainTheorem2} and \ref{MainTheorem3}, in Section~\ref{section4}.

A sequel paper is under consideration by the authors to generalize the results of Section~\ref{section4} on $SL_2(\C)$ to $n$-dimensional special group $SL_n(\C)$ for arbitrary $n$. 


\section{Preliminaries}\label{section2}
The content of this section can be found more detailed in \cite{BroMed08, CheEbi96, ColMcG93, LauPicVan13, Tho14}.


\subsection{Lorentzian manifolds}\label{lorentz-manifolds}
Here we give a brief review of some basic notations in Lorentzian geometry, even though most of them are defined in the general setting of semi-Riemannian manifolds. In this section we take $(M,\mg)$ to be a (time-orientable) Lorentzian manifold.

According to \cite{BeeEhrEas96}, $(M , \mg)$ is said to be \textit{strongly causal} if each $p \in M$ has arbitrarily small neighborhoods such that no
causal curve crosses any of these neighborhoods more than once. 

A causal curve $\gamma: [0,b) \to M$ (i.e., $\gamma$ is non-spacelike) is \textit{future directed} if for every $t\in I$, $\gamma'(t)$ is tangent to the future cone in $T_{\gamma(t)}M$.  The curve $\gamma$ is said to be \textit{future imprisoned} in a compact subset $L \subset M$ if there exists some $0<t_0<b$, such that $\gamma([t_0,b)) \subset L$. It is  \textit{partially future imprisoned} in $L$ if $\gamma(t_m) \in L$ for some increasing sequence $t_m \nearrow b$ in $[0,b)$.

\begin{proposition}[\cite{BeeEhrEas96}]\label{imprisoned}
If $(M,\mg)$ is strongly causal, then no inextendible causal curve can be partially future imprisoned in any compact
set.
\end{proposition}

A smooth curve $\gamma: I \to M$ is a \textit{geodesic} if it satisfies the equation $\nabla_{\gamma'} \gamma' = 0$, where $\nabla$ stands for the Levi-Civita connection of $\mg$. If $\gamma$ is a geodesic then $\mg(\gamma(t) , \gamma(t))$ is constant for all $t \in I$.

According to \cite{CanSan08}, a geodesic $\gamma:[0,b) \to M$, $b<+\infty$, is extendible beyond $b$ if and only if $|\gamma'|_R = \mg_R(\gamma'(t) , \gamma'(t))$ is bounded for some, hence any, complete Riemannian metric $\mg_R$, equivalently, if there exists an increasing sequence $t_m \nearrow b$ such that $\{\gamma'(t_m)\}$ converges in $TM$. 
 
A vector field $X\in \X(M)$ is called spacelike, timelike, lightlike or causal, if $X_p$ has that characteristic for every $p \in M$. A \textit{Killing} vector field on $M$ is a vector field $X$ such that $\L_X\mg = 0$, where $\L_X$ is the Lie derivation along $X$. Equivalently, $X$ is Killing if its local flows are isometries. If $X$ is Killing and $\gamma:I\to M$ is a geodesic then $\mg(X_{\gamma(t)} , \gamma'(t))$ is constant for all $t \in I$. 

Hereafter any Killing vector field will be non-zero.


\subsection{Limit curves}\label{limit-curves}
For a sequence $\{\gamma_n\}$ of smooth curves in a smooth manifold $M$, the notion of a limit curve is defined as follows: a curve $\gamma$ in $M$ is called a \textit{limit curve} of $\{\gamma_n \}$ if there exists a subsequence $\{\gamma_m \}$ such that every neighbourhood of each point $p \in \gamma$ intersect all, but possibly finite number, of the curves in $\{\gamma_m\}$ (\cite{BeeEhrEas96},~Ch.~3,~3.28).  

In semi-Riemannian manifolds, a geodesic is uniquely determined by its initial velocity. This fact leads to the existence of a distinguished limit curve for a sequence of geodesics. 
\begin{proposition}[\cite{RomSan93}]\label{limit-geodesic}
Let $(M , \mg)$ be a Lorentzian manifold, and for $m = 1, 2, \ldots$, $\gamma_m : [0 , b_m) \to M$, $b_m \leq +\infty$, be a sequence of geodesics such that $\gamma_m$ is inextendible beyond $b_m$. If $\gamma'_m(0)$ converges to $x$ in $TM$, then the geodesic in $M$ with initial velocity $x$ is a limit curve of $\{ \gamma_m \}$.
\end{proposition}
 

\subsection{Semisimple Lie algebras and the Euler equation}\label{semisimple-algebras-euler-eq}

Let $G$ be a (real) Lie group and $\g{g}$ denotes its Lie algebra which is the vector space $\X_L(G)$ of left-invariant vector fields on $G$, equipped with the commutator bracket. Given vector fields $X, Y, Z, \ldots \in \X_L(G)$, we use the lowercase $x, y, z, \ldots$ to denote the corresponding element $X_e, Y_e, Z_e, \ldots \in T_eG \cong \g{g}$. Since left-invariant vector fields are uniquely determined by their values at the identity, throughout the paper we identify any $X, Y, Z, \ldots \in \X_L(G)$ with $x, y, z, \ldots$ and use them interchangeably whenever appropriate.

Cartan's criterion states that $\g{g}$ is semisimple if and only if its Killing form $\K:\g{g}\times \g{g} \to \R$ defined by $\K(x,y)=\tr(\ad_x \circ \ad_y)$ is non-degenerate. The Killing form is a symmetric bilinear form invariant under all automorophisms of $\g{g}$. Another equivalent condition on $\g{g}$ to be semisimple is that $\g{g}$ has no no-zero abelian ideal. If $\g{g}$ has no non-zero ideal at all, then it is by definition a simple Lie algebra. The following theorem, which is a special case of Corollary~2.3 in \cite{Ver88}, states that, in general, a real simple Lie algebra does not have `large' subalgebras.
\begin{theorem}[\cite{Ver88}]\label{simple-c1-subalgebra}
Let $\g{g}$ be a real simple Lie algebra. Then $\g{g}$ has a codimension one subalgebra if and only if it is isomorphic to $\g{sl}_2(\R)$.
\end{theorem}

For the rest of the section $\g{g}$ is a semisimple Lie algebra unless otherwise is stated. The Killing form allows us to identify $\g{g}$ and its dual $\g{g}^*$ by using the correspondence $x\mapsto \K(x,\cdot)$. Here $\ad:\g{g} \to \g{gl}(\g{g})$ defined by $\ad_x(z) := [x,z]$, is the \textit{adjoint representation} of $\g{g}$ on itself. The adjoint representation of the corresponding Lie group, $\Ad : G \to Aut(\g{g})$, is given by $\Ad_g(X) = gXg^{-1}$, when $\g{g}$ is given as $\g{gl}(E)$ for some (complex) linear space $E$.

An element $x \in \g{g}$ is called \textit{semisimple}, respectively \textit{nilpotent}, if the adjoint operator $\ad_x$ is semisimple, respectively nilpotent. Recall that $\ad_x$ is semisimple if it is diagonalizable; it is nilpotent if $\ad_x^m = 0$ for some positive integer $m$. Moreover, an element $x \in \g{g}$ is called \textit{compact} if the one-parameter subgroup $\Ad(\exp(tx))$ lies in a commutative compact subgroup of $\Ad(G)$. Equivalently, $x$ is compact if the eigenvalues of $\ad_x$ are all pure imaginary.

The set of nilpotents in $\g{g}$, denoted by $\N$, is invariant under the adjoint action of $G$ on $\g{g}$ defined by $g\cdot x = \Ad_g(x)$. So $\N$ is the union of all nilpotent orbits of the adjoint action. We refer to $\N$ as the \textit{nilpotent cone}, for it clearly contains the radial line $\R x$ for any $x\in \N$. One can see that the tangent space to any orbit $O(x) \subset \N$ is given by
\begin{equation}\label{eq00}
T_x(O(x)) = T_x \Ad_G(x) = \{\ad_y(x) : y \in \g{g} \} = [x , \g{g}], \ \ \forall x \in \N .
\end{equation}

Let $\mg$ be a left-invariant semi-Riemannian metric on $G$. Using the correspondence between left-invariant tensor fields on $G$ and tensors on its tangent space at the identity element, $T_eG \cong \g{g}$, one may think of $\mg$ as a non-degenerate symmetric bilinear form on $\g{g}$. Then, associated to $\mg$, there is a unique $\K$-symmetric isomorphism $A_{\mg}$ on $\g{g}$ such that $\mg(x,y)=\K(x,A_{\mg}y)$ for all $x, y \in \g{g}$. We denote by $\mg^*$ the induced bilinear form on $\g{g}^*$. Identifying $\g{g}$ and $\g{g}^*$ as above, the isomorphism associated to $\mg^*$ is $A_{\mg}^{-1}$, that is, $\mg^*(x,y)= \K(x,A_{\mg}^{-1}y)$ for all $x, y \in \g{g}$. 

We denote by $\Lambda^*_{\mg}$ the null cone determined by $\mg^*$ in $\g{g}$:
\[
\Lambda^*_{\mg} = \{x \in \g{g} : \mg^*(x,x) = 0, \ \ x\neq 0 \}  = \{x \in \g{g} : \K(x,A_{\mg}^{-1}x) = 0,  \ \ x\neq 0 \} .
\]
The null cone is a hypersurface of $\g{g}$ and its tangent space at $x \in \Lambda^*_{\mg}$ is given by
\begin{equation}\label{eq01}
T_x \Lambda^*_{\mg} = \{y \in \g{g} : \mg^*(x,y) = 0 \}.
\end{equation}

From \eqref{eq00} and \eqref{eq01} and using $\K$-symmetry of $A_{\mg}^{-1}$ one can see that the two cones $\N$ and $\Lambda^*_{\mg}$ are transversal at $x \in \Lambda^*_{\mg} \cap \N$ if and only if $[x,A_{\mg}^{-1}x] \ne 0$. 

The position of the cones $\N$ and $\Lambda^*_{\mg}$ with respect to each other is related to the metric completeness:
\begin{proposition}[\cite{Tho14}]\label{disjoint-cones-complete}
Let $\mg$ be a left-invariant semi-Riemannian metric on a semisimple Lie group $G$. If $\N$ and $\Lambda^*_{\mg}$ are disjoint, then $\mg$ is complete.
\end{proposition}
 
Given a curve $\gamma(t)$ in $G$, one can define a curve $u(t)$ in $\g{g}$ by $u(t):=(dL_{\gamma(t)})^{-1}(\gamma'(t))$ where for each $g\in G$, $L_g$ is the left translation by $g$ in $G$. We may sometimes refer to $u(t)$ as the \textit{reflection} of $\gamma(t)$ in $\g{g}$, or, call $\gamma(t)$ the reflection geodesic of $u(t)$. For any vector field $Y$ along $\gamma$ the covariant derivative $\nabla_{\gamma'(t)}Y(t)$ of $Y$ with respect to the Levi-Civita connection $\nabla$ of $\mg$ on $G$, and the differential of the curve $v(t)=(dL_{\gamma(t)})^{-1}(Y(t))$ in $\g{g}$ are related by the following formula:
\begin{equation}\label{eq1}
(dL_{\gamma(t)})^{-1}\left(\nabla_{\gamma'(t)}Y(t)\right)=v'(t) + \nabla_{u(t)} v(t) . 
\end{equation}
where, for $x, y \in \g{g}$, $\nabla_x y$ is defined by considering $x$ and $y$ as left-invariant vector fields on $G$ and then evaluating $\nabla_x y$ at the identity element. So, in particular, the curve $\gamma(t)$ is a geodesic of $G$ (w.r.t $\mg$) if and only if $u'(t)=-\nabla_{u(t)}u(t)$. Using \eqref{eq1} and the connection formula (\cite{CheEbi96}, Ch.~3,~3.18) given by 
\begin{equation}\label{eq2}
\nabla_x y =\frac{1}{2}\left\lbrace [x,y]-(\ad_x)^*y-(\ad_y)^*x \right\rbrace, \quad x, y \in \g{g} ,
\end{equation}
with $(\ad_x)^*= - A_{\mg}^{-1} \circ \ad_x \circ A_{\mg}$ being the transpose of $\ad_x$ w.r.t $\mg$, and a change of variable, the equation $u'(t)=-\nabla_{u(t)}u(t)$ becomes, 
\begin{equation}\label{euler-equation}
u'(t) = \left[u(t),A_{\mg}^{-1}u(t)\right] .
\end{equation}
Thus, $\gamma(t)$ is a geodesic in $G$ if and only if $u(t)$ and $A_{\mg}^{-1}u(t)$ satisfy \eqref{euler-equation}. Equation \eqref{euler-equation} is referred to as the \textit{Euler equation} and, accordingly, the vector field on $\g{g}$ defined by $F_{\mg}: x\mapsto  [x,A_{\mg}^{-1}x]$, which is a homogeneous quadratic vector field, is called the \textit{Euler field}.  So completeness of the metric $\mg$ can be reformulated as follows.
\begin{theorem}[\cite{AlePut90}]\label{completeness-equivalences}
Let $(G,\mg)$ be a semi-Riemannian semisimple Lie group. Then, the followings are equivalent:
\begin{itemize}
    \item[(i)] the metric $\mg$ is (geodesically) complete,
    \item[(ii)] solutions of the Euler equation \eqref{euler-equation} are complete,
    \item[(iii)] the Euler field $F_{\mg}$ is complete.
\end{itemize}
In particular, if solutions of \eqref{euler-equation}, or equivalently, integral curves of $F_{\mg}$ are bounded, then $\mg$ is complete.
\end{theorem}

Having Theorem~\ref{completeness-equivalences} and implementing the equivalent conditions mentioned in Subsection~\ref{lorentz-manifolds} on a geodesic to be complete, one gets:
\begin{proposition}\label{completeness-criterion}
Let $(G,\mg)$ be a semi-Riemannian semisimple Lie group. Suppose that $u:[0,b) \to \g{g}$, $b<+\infty$, is a solution of the Euler equation. Then, the followings are equivalent:
\begin{itemize}
    \item[(i)] $u$ is extendible beyond $b$;
    \item[(ii)] $\|u(t)\|$ is bounded for some, hence any, norm $\| , \|$ obtained from a positive definite scalar product on $\g{g}$; 
    \item[(iii)] there exists a sequence $t_m \nearrow b$ in $[0,b)$ such that $\{u(t_m)\}$ converges in $\g{g}$.
\end{itemize}
\end{proposition}

In terms of dynamical systems, if $v' = f(v)$ is a dynamical system, where $f:U\subset \R^n \to \R^n$ is a $C^1$ map, and $v(t)$ is a trajectory of the system, then a point $y \in \R^n$ is called an $\omega$-\textit{limit point} of $v(t)$ if there exists a diverging increasing sequence $\{t_m\}$ such that $v(t_m) \to y$. In particular, if $v(t)$ has no $\omega$-limit point, then it is unbounded. The $\omega$-\textit{limit set} of a system is the set of all its  $\omega$-limit points. So the solution $u(t)$ of the Euler equation in Proposition~\ref{completeness-criterion} is extendible if and only if, it has an $\omega$-limit point. 

In the proof of Proposition~\ref{spacelike-killing-sl2}, we project the Euler equation on a hypersurface to obtain a linear system. The $\omega$-limit set of a linear system is computable, \cite{Hai08}. An special case is the following.
\begin{proposition}[\cite{Hai08}]\label{omega-limit-set}
Let $X' = AX$ be a linear dynamical system on $\R^n$ with $A\in \R^{n\times n}$. If an eigenvalue of $A$ has positive real part, then the $\omega$-limit set of the system is empty.
\end{proposition}

An \textit{idempotent} of the Euler field $F_{\mg}$ is by definition an element $x\in \g{g}$ such that $F_{\mg}(x)= x$. If $\gamma$ is a geodesic in $G$ starting at the identity so that $\gamma'(0)$ is an idempotent, then $\gamma$ runs on a one-parameter subgroup of $G$, namely, $\gamma$ is a reparametrization of $\exp(tx)$. It turns out that such a geodesic is always incomplete. Indeed, $\gamma$ is explicitly given as follows:
the integral curve $u(t)$ of the Euler field $F_{\mg}$ with initial condition $u(0) = x$ is $u(t) = \alpha(t)x$, where $\alpha(t) = 1/(1-t)$. The corresponding geodesic of $u(t)$ in $G$ is $\gamma(t) = \exp(\beta(t)A_{\mg}^{-1}x)$ where $\beta$ satisfies $\beta'(t) = \alpha(t)$ and $\beta(0) = 1$. Clearly, $\gamma(t)$ is incomplete because $u(t)$ is incomplete. 

Any semisimple Lie algebra together with Lie Poisson bracket obtained from its Lie bracket is a (linear) Poisson manifold and the Euler equation \eqref{euler-equation} is a Hamiltonian system of differential equations with Hamiltonian function $\mg^*(x,x)=\K(x,A_{\mg}^{-1}x)$. A function $f:\g{g} \to \R$ is called a \textit{first integral} or, as physicists wish to call, a \textit{constant of the motion} for the Euler equation if $f$ is constant on each of its solutions. The functions $\tr(\ad^m_x)$, for $m=1, 2, \ldots$, and $\mg^*(x,x)$ are first integrals of \eqref{euler-equation}.


\section{Lorentzian Simple Lie groups}\label{section3}

It is well known that the isometry group of a semi-Riemannian manifold is a Lie group whose Lie algebra is anti-isomorphic to the Lie algebra of complete Killing vector fields on the manifold (\cite{ONe83},~Prop.9.33). On a Riemannian simple Lie group $(G,\mg)$ any Killing vector field $Z$ can be written as $Z = Z_L + Z_R$ with $Z_L$ a left-invariant and $Z_R$ a right-invariant vector field on $G$, \cite{Gor80, OchTak76}. Any right-invariant vector field on a semi-Riemannian Lie group is Killing as its flow is given by left translations of one-parameter subgroups (\cite{Tor12},~p.~257). So, in this case, the set of left-invariant Killing vector fields determines how rich is the supply of all Killing vector fields. In non-Riemannian case, an analogous result \cite{DAm88} states that if $(G,\mg)$ is a compact Lorentzian simple Lie group then $Iso^o(G,\mg) \subset G \times G$, implying that a Killing vector field on $G$ admits a similar decomposition as a sum of a pair of left- and right- invariant vector fields. Even though a similar decomposition for Killing vector fields does not hold in general on an arbitrary semi-Riemannian Lie group, knowing the effects that the existence of left-invariant Killing vector fields can be of interest. For instance, as we will see, having a left-invariant Killing vector on a simple Lorentzian Lie group, one can obtain an additional first integral for the corresponding Euler equation.

In this section we study geodesic completeness of simple Lorentzian Lie groups admitting a left-invariant Killing vector field, that is, when $\dim(Iso(G,\mg) \cap InnAut(G)) \geq 1$. We obtain some characterizations of left-invariant Killing vector fields and prove Theorem~\ref{MainTheorem1} via Proposition~\ref{timelike-killing} and Corollary~\ref{lightlike-strungly-causal}. 

Unless otherwise stated, in this section $G$ denotes a semisimple Lie group equipped with a left-invariant Lorentzian metric $\mg$. 

We begin with stating an equivalent algebraic condition for a left-invariant vector field on $G$ to be Killing.

\begin{lemma}\label{killing-equivalent}
A vector field $Z \in \X_L(G)$ is Killing if and only if $A_{\mg} \circ \ad_z = \ad_z \circ A_{\mg}$.
\end{lemma}
\begin{proof}
Let $Z$ be a Killing vector field. Then for every pair of left-invariant vector fields $X, Y$ on $G$ one has 
\begin{equation}\label{eq15}
0 = \L_Z \mg (X,Y) = Z\mg(X,Y) - \mg([Z,X],Y) - \mg([Z,Y],X) ,    
\end{equation}
where $\L_Z$ is the Lie tensor derivation in direction of $Z$. The function $\mg(X,Y)$ is constant, so we have $0 = Z\mg(X,Y)$ which yields $\mg(\ad_zx , y) = \mg(-\ad_zy , x)$. It then follows that 
\begin{equation}\label{eq22}
(\ad_z)^* = -\ad_z .
\end{equation}

On the other hand, one has $(\ad_z)^* = - A_{\mg}^{-1} \circ \ad_z \circ A_{\mg}$. Comparing the last two equalities, one can see that $\ad_z \circ A_{\mg} = A_{\mg} \circ \ad_z$ is equivalent to \eqref{eq15} on $\X_L(G)$. One can extend this equivalence to $\X(G)$, using a global frame on $G$ obtained from any basis for $\X_L(G)$. 
\end{proof}

Next, we show that $\K(\cdot , z)$ is a first integral. 
\begin{proposition}\label{new-first-integral}
Let $Z \in \X_L(G)$ be a Killing vector field. Then, $\K(\cdot , z)$ is a first integral of the Euler equation of $\mg$. 
\end{proposition}
\begin{proof}
We need to show that $\K(u(t),z)$ is constant for any solution $u(t)$ of the Euler equation of $\mg$. Let $y\in \g{g}$, then, we have
\[
\K([y,A_{g}^{-1}y],z) = \K(y,[A_{g}^{-1}y,z]) = \K(A_{g}^{-1}y,[y,z]) = - \K([y,A_{g}^{-1}y],z),
\]
and thus 
\begin{equation}\label{eq19}
    0 = \K([y,A_{g}^{-1}y],z), \quad \text{for all } y \in \g{g}.
\end{equation}
In particular, one gets 
\[
\frac{d}{dt}\K(u(t),z) = \K(u'(t),z) = \K([u(t),A_{\mg}^{-1}u(t)],z) = 0 ,
\]
which implies that $\K(u(t),z)$ is constant.
\end{proof}

The Lie algebra $\g{sl}_2(\R)$ has the lowest dimension among all (non-compact) simple Lie algebras and, up to isomorphism, it is the only three-dimensional simple Lie algebra. In the following proposition and the examples after it, we see that, in the presence of a left-invariant Killing vector field, the behavior of a left-invariant Lorentzian metric in dimension three is different from higher dimensions. So, we consider dimension three separately and then for the rest of the section assume that $\dim(G) > 3$.

We show that in dimension three, existence of a Killing vector field implies completeness of the metric.  

\begin{proposition}\label{killing-sl2r}
Let $(G , \mg)$ be a $3$-dimensional Lorentzian simple Lie group. If there exists a left-invariant Killing vector field on $G$, then $\mg$ is geodesically complete. 
\end{proposition}
\begin{proof}
Since the geodesic equation is presented in terms of the Euler equation in the Lie algebra $\g{g} = \g{sl}_2(\R)$, we may, for simplicity, take $G = SL_2(\R)$ as the corresponding Lie group of $\g{sl}_2(\R)$.

Let $Z \in \X_L(SL_2(\R))$ be a Killing vector field. If $\mg$ is not complete, then according to \cite{BroMed08} there exists a nilpotent $0 \neq x \in \g{sl}_2(\R)$ such that $[x , A^{-1}_{\mg}x] = x$. One can take $y, \xi \in \g{sl}_2(\R)$ so that $\{x, y, \xi\}$ is an $\g{sl}_2$-triple, that is,
\[
[\xi , x] = 2x, \quad [\xi , y] = -2y , \quad [x , y] = \xi .
\]

One can use the above bracket relations to check that $\{x, y, \xi\}$ is a pseudo-orthogonal basis for $\g{sl}_2(\R)$ with respect to the Killing form $\K$. By \eqref{eq19} one gets $\K(z,x) = \K(z , [x , A^{-1}_{\mg}x]) = 0$, yielding $z = ax + b\xi$, for some $a, b \in \R$, and $[z,x] = 2bx$. Then, we have
\begin{align*}
    2bx = [z,x] & =[z , [x , A^{-1}_{\mg}x]] \\
    & = [[z , x] , A^{-1}_{\mg}x] + [x , [z , A^{-1}_{\mg}x]] \\
    & = [2bx , A^{-1}_{\mg}x] + [x , A^{-1}_{\mg}[z , x]] \\
    &= 4bx .
\end{align*}
which gives $b = 0$ and $z = ax$, implying that $0 = [z , A^{-1}_{\mg}z] = a^2[x , A^{-1}_{\mg}x] = a^2 x$. So, $a = 0$ and, consequently, $z = 0$; a contradiction.
\end{proof}

Suppose that $\{x, y, \xi\}$ is an $\g{sl}_2$-triple in $\g{sl}_2(\R)$ as in the proof of Proposition~\ref{killing-sl2r}. Let $\mg$ is the left-invariant Lorentzian metric on $SL_2(\R)$ whose associated isomorphism is given by 
\[
A_{\mg}^{-1} = 
\begin{pmatrix}
a & b & 0 \\
0 & a & 0 \\
0 & 0 & a \\
\end{pmatrix}
\]
in the basis $\{x, y, \xi\}$, for $0 \neq a, b$. Then one can see that $\ad_x$ commutes with $A_{\mg}$ and, thus, by Lemma~\ref{killing-equivalent}, the nilpotent element $x$ defines a left-invariant Killing vector field on $SL_2(\R)$. Similarly, if the metric $\mg$ is associated with  
\[
A_{\mg}^{-1} = 
\begin{pmatrix}
a & 0 & 0 \\
0 & a & 0 \\
0 & 0 & b \\
\end{pmatrix}
\]
then, the left-invariant vector field obtained from the semisimple element $\xi$ is Killing. As the following proposition shows, such examples, namely, Killing vector fields generated by nilpotent or semisimple element, exist just on $SL_2(\R)$.

We denote by $c_{\g{g}}(z)$ the centralizer of $z$ in $\g{g}$, that is, $c_{\g{g}}(z) = \{ x \in \g{g} : [z , x] = 0 \}$.  

\begin{proposition}\label{neither-semisimple-nilpotent}
Let $G$ be a simple Lie group with $\dim(G) > 3$. If $Z \in \X_L(G)$ is Killing, then, $\ad_z$ is neither semisimple nor nilpotent.
\end{proposition}
\begin{proof}
Suppose that $\ad_z$ is semisimple.  Then, one gets the decomposition $\g{g} = c_{\g{g}}(z) \oplus [z,\g{g}]$. The operator $\ad_z$ is skew-adjoint with respect to $\K$ and thus the decomposition is $\K$-orthogonal. So, in particular, $\K$ is non-degenerate on $c_{\g{g}}(z)$. Moreover, since $A_{\mg}$ is an isomorphism and commutes with $\ad_z$, it leaves both $c_{\g{g}}(z)$ and $[z,\g{g}]$ invariant.

The decomposition is also orthogonal with respect to the Lorentzian metric $\mg$. Since for $x \in c_{\g{g}}(z)$ and $y \in \g{g}$, we have 
\[
\mg(x , [z,y]) = \K(A_{\mg}x , [z,y]) = \K([A_{\mg}x,z] , y) = 0 ,
\]
which also shows that $\mg$ is non-degenerate on $c_{\g{g}}(z)$ and $[z,\g{g}]$. 

Let $\lambda$ and $\mu$ be two non-zero eigenvalues of $\ad_z$ with eigenspaces $W_\lambda$ and $W_\mu$, respectively. Then for every $x\in W_\lambda$ and $y \in W_\mu$, we have 
\begin{equation}
\lambda \mg(x,y) = \mg(\ad_zx , y) = \mg(x , \ad_z^* y) = -\mg(x , \ad_zy) = -\mu \mg(x,y) .
\end{equation}
If $\lambda = \mu$ it follows that $\mg(x,y) = 0$ and $\mg(x,x) = 0$, for any $x, y \in W_\lambda$. This shows that $W_\lambda$ is one-dimensional and also that $\mg$ is Lorentzian on $[z,\g{g}]$. 

On the other hand, if $\lambda \neq \mu$, then as above one gets $0 = \mg(x,x) = \mg(y,y)$ for any $x\in W_\lambda$ and $y\in W_\mu$, which since $\mg$ is Lorentzian yields $\mu = -\lambda$. So, either $\ad_z$ just has one non-zero eigenvalue and, therefore, $c_{\g{g}}(z)$ is a codimension one subalgebra of $\g{g}$, or it has two non-zero eigenvalues $\pm \lambda$. In the latter case one can see that $\R x \oplus c_{\g{g}}(z)$, with $\ad_zx = \lambda x$, is a subalgebra of $\g{g}$ of codimension one. Now it follows from Theorem~\ref{simple-c1-subalgebra} that $\g{g}$ is isomorphic to $\g{sl}_2(\R)$. This contradicts our assumption that $\dim(G) > 3$.

Now, suppose that $\ad_z$ is nilpotent. Then, there exists some $m \geq 3$ so that $\ad_z^{m-1} \neq 0$ and $\ad_z^m = 0$. Let $k = \frac{m}{2}$ if $m$ is even and $k = [\frac{m}{2}]+1$, when $m$ is odd. Then
\[
\mg(\ad_z^k x , \ad_z^k y) = \mg(x , (\ad_z^k)^* \circ \ad_z^k (y)) = \mg(x , (-1)^k \ad_z^{2k} y) = 0 .
\]
So, $Im(\ad_z^k)$ is a totally null subspace of $\g{g}$ and, since $\mg$ is Lorentzian, it must be a one-dimensional subspace. Note also that we have $\mg(z , \ad_z^k x) = 0$ for all $x \in \g{g}$. 

On the other hand, since $[z,A_{\mg}z] = \ad_z(A_{\mg}z) = A_{\mg}(\ad_zz) = 0$, we see that $z$ and $A_{\mg}z$ commute and, therefore, the matrix product $z.A_{\mg}z$ is nilpotent. One then gets 
\[
\mg(z,z) = \K(z , A_{\mg}z) = \tr(z.A_{\mg}z) = 0 .
\]
So, $z$ is a lightlike element which is also orthogonal to $Im(\ad_z^k)$ and, hence, $Im(\ad_z^k) = \R z$. This, in particular, yields $\ad_z^{k+1} = 0$, which according to the definition of $k$ and $m$, concludes that $m = k+1$ and, consequently, we get $k = 2$ and $m = 3$. So, we have the operator $\ad_z^2 : \g{g} \to \g{g}$ whose kernel is of codimension one. We proceed to show that $\g{g}$ has to have a codimension one subalgebra. For any $y \in \ker(\ad_z^2)$, we have $\mg(\ad_zy , \ad_zy) = -\mg(\ad_z^2y , y) = 0$ and from \eqref{eq22}, $\mg(z , \ad_zy) = - \mg(\ad_z^*z,y)=0$. This, together with the fact that $z$ is lightlike, gives $\ad_zy = \lambda z$ for some real $\lambda$. One can write $y = \lambda y_0 + x$ for some $x, y_0 \in \g{g}$, such that $\ad_zx = 0$ and $\ad_zy_0 = z$. This shows that $Im(\ad_z) = \spann_\R\{y_0 , y_1\}$ for some $y_1 \in \g{g}$ with $\ad_z^2y_1 \neq 0$. Now, one can easily see that $c_{\g{g}}(z) \oplus \R y_0$ is a codimension one subalgebra of $\g{g} = c_{\g{g}}(z) \oplus Im(\ad_z) = c_{\g{g}}(z) \oplus \spann_\R\{y_0 , y_1\}$. Again, having a codimension one subalgebra, Theorem~\ref{simple-c1-subalgebra} implies that $\g{g}$ is isomorphic to $\g{sl}_2(\R)$, which is a contradiction. 
\end{proof}

The first part of Theorem~\ref{MainTheorem1}, was proved in \cite{AlePut90}. However, for the sake of completeness, we prove it in the following proposition. Our proof is different from that of \cite{AlePut90} and the relations therein are used in the proof of other results of this section. 

\begin{proposition}\label{timelike-killing}
Let $(G , \mg)$ be a Lorentzian simple Lie group with $\dim(G) > 3$. If there exists a timelike Killing vector field $Z \in \X_L(G)$, then, solutions of the corresponding Euler equation are bounded. In particular, $\mg$ is geodesically complete.
\end{proposition}
\begin{proof}
Suppose that $Z \in \X_L(G)$ is a timelike Killing vector field. If $\mg$ is incomplete then by Theorem~\ref{completeness-equivalences} there exists an unbounded solution $u:[0,b) \to \g{g}$ of the Euler equation of $\mg$. We fix an arbitrary norm $\| , \|$ on $\g{g}$. Using Proposition~\ref{completeness-criterion}, we can choose an increasing sequence $\{t_m\}$ in $[0,b)$ converging to $b$ and satisfying $\lim_{m\to \infty}\|u(t_m)\| = \infty$ and $\lim_{m\to \infty}u(t_m)/\|u(t_m)\| = \theta$. Knowing that $\mg^*(x,x)$ is a first integral of the Euler equation, one can see that $\theta$ is nilpotent and $\mg^*(\theta , \theta) =\mg(A_{\mg}^{-1}\theta , A_{\mg}^{-1}\theta) = 0$. Lemma~\ref{new-first-integral}, provides another first integral $\K(x , z)$ which can be used to get $\K(\theta , z) =\mg(A_{\mg}^{-1}\theta , z) = 0$. So, $A_{\mg}^{-1}\theta$ is a null vector which lies in the spacelike subspace $z^\perp \subset \g{g}$; a contradiction. Consequently, all solutions of the Euler equation must be bounded which, by Theorem~\ref{completeness-equivalences}, means that $\mg$ is complete.
\end{proof}

We now turn to the case that $Z$ is lightlike.
\begin{lemma}\label{z-Az-lightlike}
Let $(G , \mg)$ be a Lorentzian simple Lie group with $\dim(G) > 3$. If there exists a lightlike Killing vector field $Z \in \X_L(G)$, such that $A_{\mg}z$ is lightlike, then, solutions of the corresponding Euler equation are bounded. In particular, $\mg$ is geodesically complete.
\end{lemma}
\begin{proof}
Suppose that $Z \in \X_L(G)$ is a Killing vector field and $z$ and $A_{\mg}z$ are both lightlike. As in the proof of Proposition~\ref{timelike-killing}, if there exists an unbounded solution of the Euler equation, then one can find $\theta \in \g{g}$ such that $\mg^*(\theta , \theta) = 0$ and $\K(\theta , z) = 0$. Furthermore, the fact that $\tr(\ad_x^m)$, $m \in \mathbb{N}$, are also first integrals of the Euler equation implies that $\theta$ is nilpotent. Since $z$ is lightlike, we have $A_{\mg}^{-1}\theta = \lambda z$ for some real $\lambda$, so $A_{\mg}z$ is also nilpotent. On the other hand, $A_{\mg}z$ and $z$ are colinear because they are both lightilke and $\mg(z,A_{\mg}z) = \K(A_{\mg}z , A_{\mg}z) = 0$. Therefore, $z$ is nilpotent, which contradicts  Proposition~\ref{neither-semisimple-nilpotent}. So, all solutions of the Euler equation are bounded and by Theorem~\ref{completeness-equivalences}, $\mg$ is complete
\end{proof}
An immediate consequence of the above proposition is the following corollary.
\begin{corollary}
Let $(G , \mg)$ be a Lorentzian simple Lie group with $\dim(G) > 3$. If there exists a lightlike Killing vector field $Z \in \X_L(G)$, such that $z$ is an eigenvector of $A_{\mg}$, then $\mg$ is geodesically complete.
\end{corollary}
 
 As we saw in Proposition~\ref{neither-semisimple-nilpotent}, semisimple and nilpotent vectors in $T_eG$ can not be extended to a left-invaraint Killing vector field on $G$. In the next result we see that when the metric is incomplete only compact elements can generate a left-invariant lightlike Killing vector field.

\begin{lemma}\label{compact-element}
Let $(G , \mg)$ be an incomplete Lorentzian simple Lie group with $\dim(G) > 3$. If $Z \in \X_L(G)$ is a lightlike Killing vector field, then $z=Z_e\in \g{g}$ is a compact element.
\end{lemma}
\begin{proof}
Recall that the centralizer $c_{\g{g}}(z)$ and $Im(\ad_z) = [z , \g{g}]$ are left invariant by $A_{\mg}$. We also know from Proposition~\ref{neither-semisimple-nilpotent} that $z$ is not nilpotent. Moreover, since $\mg$ is incomplete, as in the proof Lemma~\ref{z-Az-lightlike} one can see that $A_{\mg}z$ is nilpotent and $0 \neq \mg(A_{\mg}z , A_{\mg}z)$. Let $\xi = A_{\mg}^2z$, then $0\neq \mg(z,\xi)$, and the subspace $\spann\{\xi, z \}$ is Lorentzian. Since its orthogonal complement contains $[z , \g{g}]$, and the subspace $[z , \g{g}]$ is spacelike. Thus, for any $x \in c_{\g{g}}(z)$ and $y \in \g{g}$ we have
\[
\mg(x , [z,y]) = \K(A_{\mg}x , [z,y]) = \K([A_{\mg}x,z] , y) = 0 ,
\]
showing that $c_{\g{g}}(z)$ and $[z , \g{g}]$ are orthogonal, which together with the facts that $\dim(\g{g}) = \dim(c_{\g{g}}(z)) + \dim([z , \g{g}])$ gives the orthogonal decomposition $\g{g} = c_{\g{g}}(z) \oplus [z , \g{g}]$. 

The metric $\mg$ is positive definite on $[z , \g{g}]$ and the operator $\ad_z$, being skew-symmetric with respect to $\mg$, has just pure imaginary eigenvalues on $[z , \g{g}]$ so it is compact.
\end{proof}

The following corollary proves the second part of Theorem~\ref{MainTheorem1}.

\begin{corollary}\label{lightlike-strungly-causal}
Let $(G , \mg)$ be a Lorentzian simple Lie group with $\dim(G) > 3$. If $G$ is strongly causal and there exists a left-invariant lightlike Killing vector field $Z$ on $G$, then $\mg$ is geodesically complete.
\end{corollary}
\begin{proof}
If $\mg$ is incomplete then it follows from Proposition~\ref{compact-element} that $z$ is a compact element. That is, the one-parameter subgroup $\Ad(\exp(tz))$ is included in a compact subgroup of $\Ad(G)$. Since $G$ is simple, it is a finite covering group of $\Ad(G)$, (\cite{ColMcG93}, Ch.~1,~1.3). So one can see that the causal curve $\exp(tz)$ is also included in a compact subset of $G$ which is a contradiction according to Proposition~\ref{imprisoned} since $G$ is strongly causal. 
\end{proof}

We conclude this section with the following remark.
\begin{remark}
It can be seen from the proofs of Proposition~\ref{timelike-killing} and Lemma~\ref{z-Az-lightlike} that when there exists a left-invariant causal Killing vector field $Z$ on a Lorentzian simple Lie group $(G,\mg)$, then, $F_{\mg}(A_{\mg}(z))=[z , A_{\mg}z] = 0$. Hence, $A_{\mg}z$ is an equilibrium point of the Hamiltonian dynamical system associated with the metric $\mg$ on $\g{g}$. Moreover, every unbounded trajectory of the Euler equation is asymptotically tangent to the line spanned by $A_{\mg}z$. In particular, for those Lie groups with the property that any limit geodesic of a family of incomplete geodesics is incomplete, a left-invariant causal Killing vector field can exist only if the metric is geodesically complete.
\end{remark}


\section{Left-invariant Lorentzian Metrics on $SL_2(\C)$}\label{section4}
In this section, we prove Theorems~\ref{MainTheorem2} via Propositions~\ref{causal-killing-sl2} and \ref{spacelike-killing-sl2} and also prove Theorem~\ref{MainTheorem3}. We also provide some clarifying examples.

We begin by proving the following proposition which shows that Corollary~\ref{lightlike-strungly-causal} holds for $SL_2(\C)$ without assuming $(SL_2(\C),\mg)$ to be strongly causal. 

\begin{proposition}\label{causal-killing-sl2}
Let $\mg$ be a left-invariant Lorentzian metric on $SL_2(\C)$. If there exists a left-invariant causal Killing vector field on $SL_2(\C)$, then, solutions of the Euler equation are bounded and, in particular, $\mg$ is geodesically complete.
\end{proposition}
\begin{proof}
Let $Z \in \X_L(SL_2(\C))$ be a Killing vector field such that $\mg(z,z) \leq 0$. Having Proposition~\ref{timelike-killing}, we just need to consider the case $\mg(z,z) = 0$. In this case, as in the proof Lemma~\ref{z-Az-lightlike}, it follows that if there exists an unbounded solution for the Euler equation, then $A_{\mg}z$ is nilpotent. On the other hand, since, by Lemma~\ref{killing-equivalent}, $A_{\mg} \circ \ad_z = \ad_z \circ A_{\mg}$, we get $[z , A_{\mg}z] = 0$ which means $A_{\mg}z = \lambda z$ for some $\lambda \in \C$. So, $z$ is also nilpotent; contrary to Proposition~\ref{neither-semisimple-nilpotent}. Hence, all solutions of the Euler equation are bounded and $\mg$ is geodesically complete by Theorem~\ref{completeness-equivalences}.
\end{proof}

Before we continue with spacelike case, we show that given a Lorentzian metric $\mg$ on a semisimple Lie group $G$, the adjoint action of $G$ provides a family of Lorentzian metrics on $G$ with the same behavior as the metric $\mg$ in geodesic completeness. 

Suppose that $\mg$ is a left-invariant metric on a semisimple Lie group $G$. For each $g\in G$, $\mg_g(x,y):=\K(x,\Ad_g \circ A_{\mg} \circ \Ad_{g^{-1}} (y))$ is a Lorentzian bilinear form on $\g{g}$ and induces a Lorentzian left-invariant metric on $G$.  

\begin{lemma}\label{AdG-family}
\begin{itemize}
\item[(i)] For any $g\in G$, the metric $\mg_g$ is complete if and only if $\mg$ is complete.
\item[(ii)] An element $x \in \g{g}$ is an idempotent or a zero of the Euler equation of $\mg$ if and only if $\Ad_g x$ is, respectively, an idempotent or a zero of the Euler equation of $\mg_g$.
\end{itemize}
\end{lemma} 
\begin{proof}
Suppose that $u(t)$ is a solution of the Euler equation of $\mg$, that is $[u(t) , A_{\mg}^{-1}u(t)] = u'(t)$, and put $v(t) = \Ad_gu(t)$. Then, we have
\begin{align}\label{eq8}
v'(t) = \Ad_gu'(t) & = \Ad_g[u(t) , A_{\mg}^{-1}u(t)] \nonumber \\
& = [\Ad_gu(t) , \Ad_g \circ A_{\mg}^{-1} \circ \Ad_{g^{-1}}(\Ad_gu(t))] \nonumber \\
& = [v(t) , \Ad_g \circ A_{\mg}^{-1} \circ \Ad_{g^{-1}} (v(t))].
\end{align}

Since $\Ad_g \circ A_{\mg} \circ \Ad_{g^{-1}}$ is the associated isomorphism of $\mg_g$, it follows from \eqref{eq8} that the Euler equation of $\mg$ has complete solutions if and only if solutions of the Euler equation of $\mg_g$ are complete So part (i) follows from Theorem~\ref{completeness-equivalences}. 

For part (ii), if $[x , A_{\mg}^{-1} x]$ is either $0$ or $x$, then one can see from \eqref{eq8} that $[\Ad_gx , \Ad_g \circ A_{\mg}^{-1} \circ \Ad_{g^{-1}}(\Ad_gx)]$ is either $0$ or $\Ad_gx$, respectively. The converse follows from $\mg = (\mg_g)_{g^{-1}}$. 
\end{proof}

As mentioned in Subsection~\ref{semisimple-algebras-euler-eq}, an idempotent of the Euler field $F_{\mg}$ is an element $x \in \g{g}$ which is the initial velocity of a non-constant geodesic $\gamma$ starting at the identity and running on a one parameter subgroup of $G$. In this case, the graph of the reflected curve $u(t)$ of $\gamma(t)$ is just the radial half-line $\R^+ x$.

To establish our next result, we introduce a family of curves which includes the above mentioned reflected curves as an special case.

\begin{definition}
Let $F$ be a homogeneous quadratic vector field on a vector space $\V$, and $u(t)$ be an integral curve of $F$. Then, $u(t)$ is called a generalized conical spiral (GCS) of $F$, if there exist $r, s \in \R$ with $r>0$, so that $u(s) = ru(0)$.
\end{definition}

Note that since $F$ is homogeneous quadratic, one gets $u(ks) = r^k u(0)$ for any integer $k$. Here are some features of GCS curves.
\begin{proposition}
Let $F$ be a homogeneous quadratic vector field on a vector space $\V$. Let $x\in \V$ and $u:I \subset \R \to \V$ be a GCS of $F$ with $u(s) = ru(0)$ for some $r, s \in \R$. Then, the followings hold:
\begin{itemize}
    \item[(i)] if $F(x)=x$; $\R x$ is a GCS,
    \item[(ii)] if $r = 1$; $u(t)$ is an ordinary $s$-periodic curve and, therefore, complete,
    \item[(iii)] if $r \neq 1$; $u(t)$ is incomplete.
\end{itemize}
Moreover, when $\V=\g{g}$ is a Lie algebra with Lie group $G$ and $F=F_{\mg}$ is the Euler field of a left-invariant Lorentzian metric $\mg$ on $G$. Then,
\begin{itemize}
\item[(iv)] if $r \neq 1$; $u(t)$ is nilpotent and $\mg^*(u(t) , u(t)) = 0$, for any $t\in I$.
\end{itemize}
\end{proposition}
\begin{proof}
Parts (i) and (ii) are trivial. For part (iii), let $r>1$ and $[0,b)$ be the maximal domain of the curve $u(t)$. For each non-negative integer $k$, define $\tilde{u}:=u\big|_{[0,s]}$ and let $s_k = \sum_{j=0}^{k-1}\frac{s}{r^j}$. Then one gets
\[
u(t) = r^k \tilde{u}(r^k(t-s_k)), \quad t\in [s_k , s_k + \frac{s_k}{r^k}].
\]
This yields $b = sr/(r-1)$. So, $u$ is incomplete.

Similarly, if $r<1$ then one can see that $u(t)$ can not be extended beyond a finite time in negative direction. 

Finally, to prove (iv), let $u(t)$ be a GCS curve of the Euler field $F_{\mg}(x) = [x , A_{\mg}x]$. We have $\mg^*(u(s) , u(s)) = r^2 \mg^*(u(0) , u(0))$. On the other, since $\mg^*(x,x)$ is a first integral, one gets $\mg^*(u(s) , u(s)) = \mg^*(u(0) , u(0))$. Putting together, we conclude that $(r^2-1)\mg^*(u(0) , u(0))=0$, since $r \neq 1$, so $\mg^*(u(0) , u(0))=0$. Since  $\mg^*(u(t) , u(t))$ is constant, we get $\mg^*(u(t) , u(t))=0$ for any $t\in I$.

In a similar way, one can use the first integral $\tr(\ad_{u(t)}^m)$ to show that $u(t)$ is nilpotent.
\end{proof}

We are now in a position to prove Theorem~\ref{MainTheorem2} for a spacelike vector field $Z$. 
\begin{proposition}\label{spacelike-killing-sl2}
Let $\mg$ be a left-invariant Lorentzian metric on $SL_2(\C)$. If there exists a Killing vector field $Z\in \X_L(SL_2(\C))$, then $\mg$ is geodesically complete if and only if its Euler field has no GCS.
\end{proposition}

We break the proof of Proposition~\ref{spacelike-killing-sl2} into several steps. In Lemma~\ref{lema1-spacelike=sl2c} we obtain convenient representations for $A_{\mg}^{-1}$ and $\ad_Z$. We then consider solutions of the Euler equation in the zero level set of the first integral $\K(x,z)$ and show that these solutions project onto solutions of a linear system on $Im(\ad_z)$, in Lemma~\ref{lema2-spacelike=sl2c}. In Lemmas~\ref{lema3-spacelike=sl2c} and \ref{lema4-spacelike=sl2c} we examine completeness of the Euler field on $\g{sl}_2(\C)$ by investigating the linear system on $Im(\ad_z)$.

In the following lemmas we assume the setting of Proposition~\ref{spacelike-killing-sl2}. 

\begin{lemma}\label{lema1-spacelike=sl2c}
Let $c_{\g{g}}(z)$ be the centralizer of $z$ in $\g{\g{sl}_2(\C)}$. 
\begin{itemize}
    \item[(i)] $\g{sl}_2(\C)$ is decomposed orthogonally as $c_{\g{g}}(z) \oplus [z , \g{sl}_2(\C)]$. This decomposition is preserved by $A_{\mg}$ and $\mg$ is non-degenerate on each factor.
    \item[(ii)] The metric $\mg$ is positive definite on $[\xi , \g{sl}_2(\C)]$ and in an orthonormal basis for $[\xi , \g{sl}_2(\C)]$, $\ad_z$ and $A_{\mg}^{-1}$ can be represented as:
\begin{equation}\label{eq18}
\ad_z\big|_{[\xi , \g{g}]} =
  \left(
  \begin{array}{@{}c|c@{}}
  \begin{array}{@{}cc@{}}
     0  &  -c \\
     c  &  0
  \end{array} & 0 \\ \hline
  0 & \begin{array}{cc}
      0  & -c \\
      c &  0
  \end{array}
  \end{array}
  \right), \quad 
  A_{\mg}^{-1}\big|_{[\xi , \g{g}]} = 
    \left(
  \begin{array}{@{}cccc@{}}
   a &  &  & 0 \\
    & a &  &  \\
    &  & b &  \\
   0 &  &  & b
  \end{array}
  \right), \ \ ab < 0 .
\end{equation}
 Moreover, $A_{\mg}^{-1}$ on $c_{\mg}(z)$ can be written as  
 \begin{equation}\label{eq21}
  A_{\mg}^{-1}\big|_{c_{\g{g}}(\xi)} = \begin{pmatrix} 
  d_1 & d_2 \\ -d_2 & d_3 \end{pmatrix} , \quad d_1<0 .
  \end{equation}
\end{itemize}
\end{lemma}
\begin{proof}
Let $Z \in \X_L(SL_2(\C))$ be a spacelike Killing vector field, that is, $\mg(z,z) > 0$. From Proposition~\ref{neither-semisimple-nilpotent} we know that $z$ is not nilpotent which means that $\tr(\ad_z^2) \neq 0$. Using Lemma~\ref{AdG-family} and having $\Ad_g(z) = gzg^{-1}$, one may assume that $z$ is represented by $z = \begin{pmatrix} \lambda & 0 \\ 0 & -\lambda \end{pmatrix}$, where $\lambda \in \C$ is pure imaginary. Now it can be easily seen that $\g{sl}_2(\C)$ has the orthogonal decomposition $\g{sl}_2(\C) = c_{\g{g}}(z) \oplus [z , \g{sl}_2(\C)]$. The fact that $A_{\mg}$ leaves this decomposition invariant and $\mg$ is non-degenerate on each factor, follows from  Proposition~\ref{neither-semisimple-nilpotent}. This proves part~(i).

For part (ii), let $\xi$ denote the semisimple element $\begin{pmatrix} 1 & 0 \\ 0 & -1 \end{pmatrix} \in \g{sl}_2(\C)$. Then $z = \lambda \xi$ and one has
\[
\g{sl}_2(\C) = c_{\g{g}}(z) \oplus [z , \g{sl}_2(\C)] = c_{\g{g}}(\xi) \oplus [\xi , \g{sl}_2(\C)] ,
\]
where $c_{\g{g}}(-)$ is a simpler notation used for the centralizer $c_{\g{sl}_2(\C)}(-)$.

Note that $c_{\g{g}}(z) = c_{\g{g}}(\xi)$ is a two-dimensional subspace spanned by $\{\xi, \i \xi\}$. The Killing form $\K$ has index $(1,1)$ on $c_{\g{g}}(\xi)$ and $(2,2)$ on $[\xi , \g{sl}_2(\C)]$. 

Suppose that $0 \neq \mu$ is an eigenvalue of $A_{\mg}^{-1}$ with corresponding eigenspace $W_\mu$. For any $x\in W_\mu$ with $\mg^*(x,x) \neq 0$, we get
\[
\K(x,x) = \mu \K(x , A_{\mg}^{-1}x) = \mg^*(x,x) \neq 0 .
\]
On the other hand, since $A_{\mg}z \circ \ad_z = \ad_z \circ A_{\mg}z$, we get $A_{\mg}^{-1}[z , x] = [z , A_{\mg}^{-1}x] = \mu [z , x]$, which implies that $[z , x] \in W_\mu$. The vectors $z$ and $[z , x]$ are linearly independent, for otherwise, $x$ would be nilpotent and one gets $0 = \mu \K(x,x) = \K(x,A_{\mg}^{-1}x) = \mg^*(x,x)$. So, dimension of $W_\mu$ is at least $2$.

Now, we show that $\mg$ is positive definite on $[\xi , \g{sl}_2(\C)]$. 

One can see that the isomorphism $A_{\mg}^{-1}$ is symmetric with respect to $\mg$. Therefore, if $\mg$ is negative definite on $[\xi , \g{sl}_2(\C)]$, $A_{\mg}^{-1}$ can be represented on $[\xi , \g{sl}_2(\C)]$ in the following four matrices (\cite{ONe83}, p.~261):  
\begin{align*} 
& (i): 
  \left(
  \begin{array}{@{}cccc@{}}
   a_1 &  &  & 0 \\
    & a_2 &  &  \\
    &  & a_3 &  \\
   0 &  &  & a_4
  \end{array}
  \right), & 
& (ii): 
  \left(
  \begin{array}{@{}c|c@{}}
  \begin{array}{@{}cc@{}}
     a  &  b \\
     -b  &  a
  \end{array} & 0 \\ \hline
  0 & \begin{array}{cc}
      c  &  \\
       &  d
  \end{array}
  \end{array}
  \right) , \\  
&(iii): 
  \left(
  \begin{array}{@{}c|c@{}}
  \begin{array}{@{}cc@{}}
     a  &  0 \\
     b  &  a
  \end{array} & 0 \\ \hline
  0 & \begin{array}{cc}
      c  &  \\
       &  d
  \end{array}
  \end{array}
  \right) , & 
&(iv):  
  \left(
  \begin{array}{@{}cccc@{}}
   a & 0 & 1 & 0 \\
   0 & a & 0 & 0 \\
   0 & 1 & a & 0 \\
   0 & 0 & 0 & b
  \end{array}
  \right) 
\end{align*}
where in cases $(i)$ and $(ii)$, $A_{\mg}^{-1}$ is represented in an orthonormal basis, while in cases $(iii)$ and $(iv)$ it is represented with respect to a pseudo-orthonormal basis. 

As we saw above, the eigenspace of any non-zero eigenvalue of $A_{\mg}^{-1}$ has dimension at least two. So, in case $(i)$ we can assume $a_1=a_2$ and $a_3 = a_4$. Now, one can use $\K(x,y) = \mg(x,A_{\mg}^{-1}y)$ to see that if $a_1 a_3 > 0$, then the index of $\K$ on $[h,\g{sl}_2(\C)]$ would be $(1,3)$ and if $a_1 a_3 < 0$, it would be $(3,1)$. This eliminates case $(i)$ as a possible matrix form of $A_{\mg}^{-1}$. Case $(ii)$ can be excluded similarly, by assuming $c=d$ and then calculating the index of $\K$.

In case $(iii)$, we have three eigenvalues and knowing that their eigenspaces can not be of dimension $1$, gives us that $a=c=d$. Let the pseudo-orthonormal basis be $\{\eta_1, \eta_2, e_1, e_2\}$ with $\mg(\eta_1,\eta_1) = \mg(\eta_2,\eta_2) = 0$ and $\{e_1, e_2\}$ orthonormal, then, the index of $\K$ on $\spann\{\eta_1, \eta_2\}$ is $(1,1)$ and on $\spann\{e_1, e_2\}$ is either $(0,2)$ or $(2,0)$. So, case $(iii)$ can be dropped, as well. A similar argument can dismiss case $(iv)$.

Therefore, the metric $\mg$ is positive definite on $[\xi , \g{sl}_2(\C)]$ and, considering the index of $\K$, $A_{\mg}^{-1}$ has the representation $\diag(a, a, b, b)$ with $ab < 0$, in an orthonormal basis $\{e_1, \ldots, e_4\}$ for $[\xi , \g{sl}_2(\C)]$. Therefore, for vectors $e_1 + e_3$ and $e_1 + e_4$ one gets
\[
[e_1 + e_3 , A_{\mg}^{-1}(e_1 + e_3)] = (b-a)[e_1 , e_3], \quad 
[e_1 + e_4 , A_{\mg}^{-1}(e_1 + e_4)] = (b-a)[e_1 , e_3].
\]

From\eqref{eq19} and knowing that for any $x, y \in [\xi , \g{sl}_2(\C)]$, $[x , y] \in c_{\g{g}}(\xi)$, we get $0 = \K(z , [e_1 , e_3]) = \K(z , [e_1 , e_4])$. This fact with non-degeneracy of $\K$ on $c_{\g{g}}(\xi)$, imply that $[e_1 , e_3]$ and $[e_1 , e_4]$ are linearly dependent. Thus, their Lie bracket is zero and it follows that $\mu e_1 \in \spann\{e_3 , e_4\}$ for some $\mu \in \C$. But $\{e_1, \ldots, e_4\}$ is orthonormal, so we must have $\i e_1 \in \spann\{e_3 , e_4\}$. Similarly, one gets $\i e_2 \in \spann\{e_3 , e_4\}$. So, we can use $\{e_1, e_2, \i e_2, \i e_2 \}$ as an orthonormal basis for $[\xi , \g{sl}_2(\C)]$. The operator $\ad_z$ is skew-symmetric with respect to $\mg$ and commutes with $A_{\mg}^{-1}$. So, in this basis $\ad_z$ and $A_{\mg}$ admit the representations in \eqref{eq18}.

We have $\ad_z(e_1 + \i e_2) = \i c(e_1 + \i e_2)$ and $\ad_z(e_2 + \i e_1) = \i c(e_2 + \i e_1)$ which shows that $z$ and $\i \xi$ are linearly dependent. Thus, without loss of generality, we can assume $z = \i \xi$. The vector $\i \xi$ is spacelike, so one can write 
\[
  A_{\mg}^{-1}\big|_{c_{\g{g}}(\xi)} = \begin{pmatrix} 
  d_1 & d_2 \\ -d_2 & d_3 \end{pmatrix}  .
\]
Since $z$ is spacelike, one can assume that $d_1 < 0$.
\end{proof}

From Lemma~\ref{new-first-integral}, we know that each solution of Euler equation lies in a level set of the function $\K(x,\i \xi)$. In the following lemma we consider solutions in the zero level set $\K(x,\i \xi) = 0$. Such a solution can be written as $u(t) = f(t)\xi + v(t)$ where $f:I\subset \R \to \R$ is a smooth function and $v:I \to [\xi , \g{sl}_2(\C)]$ is the projection of $u(t)$ on $[\xi , \g{sl}_2(\C)]$.

\begin{lemma}\label{lema2-spacelike=sl2c}
The Euler equation of $\mg$ on $\K(x,\i \xi) = 0$ can be written as 
\begin{equation}\label{eq23}
[u , A_{\mg}^{-1}u] = \varphi(v) \xi + f V(v) ,
\end{equation}
where $\varphi: [\xi , \g{sl}_2(\C)] \to \R$ and $V$ is a linear operator on $[\xi , \g{sl}_2(\C)]$. Furthermore, if the system $y'= V(y)$ has just bounded solutions, then solutions of the Euler equation on $\K(x,\i \xi) = 0$ are also bounded and, thus, are complete.
\end{lemma}
\begin{proof}
Given a solution $u(t) = f(t)\xi + v(t)$ of the Euler equation, we then have
\begin{equation}\label{eq16}
[u , A_{\mg}^{-1}u] = f^2[\xi , A_{\mg}^{-1}\xi] + f[v , A_{\mg}^{-1} \xi] + f[\xi , A_{\mg}^{-1}v] + [v , A_{\mg}^{-1}v] .
\end{equation}
In the above equation one has $[\xi , A_{\mg}^{-1}\xi] = 0$, $[v , A_{\mg}^{-1}v] \in c_{\g{g}}(\xi)$ and $[v , A_{\mg}^{-1} \xi], [\xi , A_{\mg}^{-1}v] \in [\xi , \g{sl}_2(\C)]$. So, one gets $\K(\i \xi , [v , A_{\mg}^{-1}v]) = \K(\i \xi , [u , A_{\mg}^{-1}u]) = 0$, which yields $[v , A_{\mg}^{-1}v] = \varphi(v) \xi$ for some $\varphi: [\xi , \g{sl}_2(\C)] \to \R$. Thus, \eqref{eq16} becomes
\begin{equation}\label{eq17}
[u , A_{\mg}^{-1}u] = \varphi(v) \xi + f \left([v , A_{\mg}^{-1} \xi] + [\xi , A_{\mg}^{-1}v] \right) .
\end{equation}

The linear hyperspace of $\g{sl}_2(\C)$ defined by $0 = \K(x , \i \xi)$ is decomposed as $\R \xi \oplus [\xi , \g{sl}_2(\C)]$. Let $\pi : \R \xi \oplus [\xi , \g{sl}_2(\C)] \to [\xi , \g{sl}_2(\C)]$ be the projection map. We denote by $V$ the linear vector field on $[\xi , \g{sl}_2(\C)]$ obtained from the second term of \eqref{eq17}, that is, 
\[
V: [\xi , \g{sl}_2(\C)] \to [\xi , \g{sl}_2(\C)], \quad y\mapsto [\xi, (A_{\mg}^{-1} - d_1\id)y] + d_2[\i \xi , y] ,
\]
in which we used $A_{\mg}^{-1}\xi = d_1 \xi - d_2\i \xi$. In the above formula $\id : [\xi , \g{sl}_2(\C)] \to [\xi , \g{sl}_2(\C)]$ is the identity map. 

Let $\tilde{v}(t)$ be the solution of $\tilde{v}'(t) = V(\tilde{v}(t))$ with initial condition $\tilde{v}(0) = y_0$. Suppose that $u(t) = f(t)\xi + v(t)$ be the solution of the Euler equation with $u(0) = f(0)\xi + y_0$. Then, as long as $u(t)$ does not intersect the hyperspace $[\xi , \g{sl}_2(\C)]$, the curves $v(t) = \pi(u(t))$ and $\tilde{v}(t)$ have the same graph with, possibly, different parameterizations. In particular, if the solutions of $\tilde{v}'(t) = V(\tilde{v}(t))$ are bounded, then so is for the Euler equation.
\end{proof}

Let $c$ be the constant in \eqref{eq18} and $V$ the operator in \eqref{eq23}. To simplify calculations we consider $\frac{1}{c}V$ instead of $V$ in the following. The operator $\frac{1}{c}V$ has the representation 
\[
\left(
  \begin{array}{@{}cccc@{}}
   0 & -d_2 & 0 & -(b-d_1) \\
    d_2 & 0 & (b-d_1)  & 0  \\
    0 & (a-d_1) & 0 & -d_2 \\
    -(a-d_1)  & 0 & d_2 & 0
  \end{array}
\right).
\]
in the basis $\{e_1, e_2. \i e_1, \i e_2\}$, and its characteristic polynomial is given by 
\[
P_V(x) = (x^2 + d_2^2 - d)^2 + 4d_2^2d ,
\]
where $d=(d_1-b)(d_1-a)$. In the following we consider different possibilities for $d$.

\begin{lemma}\label{lema3-spacelike=sl2c}
If $d\leq 0$ then solution of the Euler equation are complete.
\end{lemma}
\begin{proof}
Suppose that $d < 0$ then $V$ only has pure imaginary eigenvalues. Let $(x_1, x_2, y_1, y_2)$ denote the coordinates in the basis $\{e_1, e_2, \i e_1, \i e_2\}$. Then, having $\K(\i \xi, x)$, $\mg^*(x , x)$ and $\tr(\ad_x^m)$ as first integrals and using $\mg^*(x,y) = \K(x , A_{\mg}^{-1}y)$ one obtains the following system of equations:
\begin{align} \label{eq20}
    x_1^2 + x_2^2 -y_1^2 - y_2^2 + f^2 = 0 \nonumber \\
    a x_1^2 + a x_2^2 -b y_1^2 - b y_2^2 + d_1 f^2  = 0 \\
    x_1 y_1 + x_2 y_2 = 0 \nonumber
\end{align}
From that, one gets $(a-d_1)(x_1^2 + x_2^2) + (d_1-b)(y_1^2 + y_1^2) = 0$, which only has zero as a solution. On the other hand, we know that if the Euler equation has an unbounded solution then the zero level set of the above three first integrals have at least one $\theta \neq 0$ in common. Hence, when $d < 0$, solutions of the Euler equation are bounded.

Now suppose that $d = 0$. The constants $a$ and $b$ in \eqref{eq18} have different signs. There is no loss of generality in assuming $a>0$ and $b<0$. One then obtains $d_1 = b$. If we write $x$ in the hyperspace $\K(x , \i \xi) = 0$ as $x = f\xi + y$ with $y = (x_1, x_2, y_1, y_2)$ representing an element of $[\i \xi , \g{sl}_2(\C)]$ in the basis $\{e_1, e_2, \i e_1, \i e_2 \}$, then one gets
\[
[x , A_{\mg}^{-1}] = f \left([y , [\xi , A_{\mg}^{-1}y] + A_{\mg}^{-1} \xi] \right) + [y , A_{\mg}^{-1}y],
\]
where $[y , A_{\mg}^{-1}y] = (a-b)(x_2y_1 - x_1y_2)$ and 
\begin{align*}
[\xi , A_{\mg}^{-1}y] + [y , A_{\mg}^{-1} \xi] = (-d_2x_2)e_1 + (d_2x_1)e_2 + ((a-b)x_2 - d_2y_2) \i e_1 \\ - ((a-b)x_1 - d_2y_1) \i e_2 .
\end{align*}

So, the Euler equation is given by
\begin{align*}
    f' & = (a-b)(x_1y_2 - y_1x_2) , & &\\
    x'_1 & = -d_2 f x_2 , & y'_1 & = (a-b) f x_2 - d_2 f y_2 , \\
    x'_2 & = d_2 f x_1 , &  y'_2 & = -(a-b) f x_1 + d_2 f y_1 . 
\end{align*}

From the above equations one gets $x'_1 = -d_2 f x_2$ and $x'_2 = d_2 f x_1$, which yields $x'_1 x_1 + x'_2 x_2 = 0$. Thus, $x_1^2 + x_2^2 = r^2$, for some constant $r$. Then from the above system we have 
\begin{align*}
\frac{1}{(a-b)}f'' & = x'_2 y_1 - y'_2 x_1 + x_2 y'_1 - y_2 x'_1 \\
& = (a-b)f(x_1^2 + x_2^2) \\
& = (a-b)r^2f ,
\end{align*}
So, for any solution $u(t) = f(t)\i \xi + v(t)$ of the Euler equation, $f(t)$ is complete and since $f$ is the coefficient in the timelike direction, the curve $v(t) = (x_1(t), x_2(t), y_1(t), y_2(t))$ also turns out to be complete.
\end{proof}

Next, we consider the case $d>0$.

\begin{lemma}\label{lema4-spacelike=sl2c}
If $d > 0$ then $\mg$ is complete if and only if $F_{\mg}$ has no GCS.
\end{lemma}
\begin{proof}
In this case the eigenvalues of $V$ are $\pm(\sqrt{d}+\i d_2)$ and $\pm(\sqrt{d}-\i d_2)$. As in the case of $d<0$, one gets the system of equations \eqref{eq20} on the hyperspace $\K(x , \i \xi) = 0$, which determines two copies of the cone $S^1 \times \R^+$ where the factor $\R$ is associated with coefficient $f$ of the timelike vector $\xi$. One can also check that the points on these cones satisfy $0 \neq [x , A_{\mg}^{-1}x]$, namely, they are in the intersection $\N \cap \Lambda^*_{\mg}$. So, if $u(t)$ is a solution of the Euler equation $u'(t)=[u(t) , A_{\mg}^{-1}u(t)]$ with $u(0) \in (\N \cap \Lambda^*_{\mg}) \cap \{ \K(x , \i \xi) = 0 \}$, then, it lies in $S^1 \times \R^+$. By Proposition~\ref{omega-limit-set}, since $V$ has an eigenvalue with positive real part, the $\omega$-limit set of the curve $\pi(u(t))$ in $[\xi , \g{sl}_2(\C)]$ is empty and $\pi(u(t))$ is distancing from the origin, in the direction of $\xi$. This means that the curve $u(t)$ in $S^1 \times \R^+$ is unbounded. So, either the projection of $u(t)$ on $S^1$ covers the whole $S^1$, or we have an idempotent. In either case $u(t)$ is an incomplete GCS. 
\end{proof}

 Lemmas~\ref{lema3-spacelike=sl2c} and \ref{lema4-spacelike=sl2c} in the light of Theorem~\ref{completeness-equivalences} complete the proof of Proposition~\ref{spacelike-killing-sl2}.

As we saw in the proof of Proposition~\ref{spacelike-killing-sl2}, the Euler field can have incomplete GCS just in the situation of Lemma~\ref{lema4-spacelike=sl2c} and such solutions must be included in $\N \cap \Lambda^*_{\mg}$. So, the corresponding geodesics in $SL_2(\C)$ are lightlike. Therefore, we have
\begin{corollary}\label{complete-lightlikecomplete}
Let $\mg$ be a left-invariant Lorentzian metric on $SL_2(\C)$. Suppose that there exists a left-invariant Killing vector field on $SL_2(\C)$. Then, $\mg$ is complete if and only if it is lightlike complete.
\end{corollary}

\begin{remark}\label{remark2}
In the proof of Lemma~\ref{lema4-spacelike=sl2c}, when $d_2 \neq 0$, one obtains an incomplete solution of the Euler equation which is not generated by an idempotent. So, unlike $SL_2(\R)$, there are incomplete left-invariant Lorentzian metric on $SL_2(\C)$ with no idempotent.   
\end{remark}

\begin{remark}\label{remark3}
In Lemma~\ref{lema1-spacelike=sl2c}, the assumption that $Z$ is left-invariant and Killing is sufficient to prove the lemma. Supposing that $Z$ is spacelike just allows us to assume $d_1 < 0$. Indeed, one can check that in \eqref{eq21}, if $d_1 > 0$ and $d_2 = 0$, then $A_{\mg}$ defines a left-invariant Lorentzian metric on $SL_2(\C)$ for which $Z = \i\xi$ is a left-invariant Killing timelike vector field. Similarly, when $d_1 = d_3 = 0$ and $d_2 \neq 0$, then the left-invariant Killing vector field $Z = \i\xi$ is lightlike. These provide explicit examples for Theorems~\ref{MainTheorem1} (i),~\ref{MainTheorem2}.
\end{remark}

In \cite{Tho14} a stronger version of Proposition~\ref{disjoint-cones-complete} was proved for $SL_2(\R)$, which states that a left-invariant Lorentzian metric on $SL_2(\R)$ is complete if $\Lambda^*_{\mg}$ and $\N$ are disjoint or tangent at some point, and it is incomplete if $\Lambda^*_{\mg}$ and $\N$ are transversal. It is then claimed that this is true for any Lorentzian semisimple Lie group. Here is a counterexample giving a left-invariant Lorentzian metric on $SL_2(\C)$ which is complete, while $\Lambda^*_{\mg}$ and $\N$ are transversal. 

\vspace{0.5cm}
\textbf{Example:}\label{example1-sl2c} Take the following complex basis for $\g{sl}_2(\C)$:
\[
e_1 = \frac{1}{\sqrt{2}} \begin{pmatrix} 1 & 0 \\ 0 & -1 \end{pmatrix}, \quad
e_2 = \frac{1}{\sqrt{2}} \begin{pmatrix} 0 & 1 \\ 1 & 0 \end{pmatrix}, \quad
e_3 = \frac{1}{\sqrt{2}} \begin{pmatrix} 0 & 1 \\ -1 & 0 \end{pmatrix}
\]

Considering $\g{sl}_2(\C)$ as the complexification of $\g{sl}_2(\R)$ by $\g{sl}_2(\C) = \g{sl}_2(\R) \oplus \i \g{sl}_2(\R)$, $\{ e_1, e_2, e_3 \}$ is a basis for $\g{sl}_2(\R)$. Let $\mg$ be the left-invariant Lorentzian metric on $SL_2(\C)$ with associated isomorphism $A_{\mg}$ whose inverse is given by the matrix $A_{\mg}^{-1} = \diag\{ 1, 1, 2, -3, -3, 1\}$ in the real basis $\{e_1, e_2, e_3, \i e_1, \i e_2, \i e_3\}$ for $\g{sl}_2(\C)$. 

Let $x = (x_1, x_2, x_3, y_1, y_2, y_3)$ denote the coordinate of $x\in \g{sl}_2(\C)$ in the above real basis. Then the null cone $\Lambda^*_{\mg}$ defined by $\mg^*(x,x) = \K(x , A^{-1}x) = 0$ is determined by the following equation
\begin{equation} \label{eq10}
x_1^2 + x_2^2 -2x_3^2 + 3y_1^2 + 3y_2^2 + y_3^2 = 0 .
\end{equation}
On the other hand, $x$ belongs to $\N$ if it satisfies $\tr_\C x^2 =0$ or, equivalently, if
\begin{align}
x_1^2 + x_2^2 -x_3^2 - y_1^2 - y_2^2 + y_3^2 = 0 , \label{eq11}\\
x_1 y_1 + x_2 y_2 - x_3 y_3 = 0 . \label{eq12}
\end{align}

For any $x \in \Lambda^*_{\mg} \cap \N$, the one-dimensional subspace $\R x \subset \g{sl}_2(\C)$ is also included in the intersection. So, to determine $\Lambda^*_{\mg} \cap \N$, we first equip $\g{sl}_2(\C)$ with an arbitrary norm, for simplicity, the square norm $\|x\| = \sum_i x_i^2 + \sum_j y_j^2$, and then determine unit vectors in $\g{sl}_2(\C)$ satisfying \eqref{eq10}-\eqref{eq12}. If $x$ is a unit vector, then \eqref{eq11} yields
\begin{equation}\label{eq13}
x_1^2 + x_2^2 + y_3^2 = \frac{1}{2} , \quad
x_3^2 + y_1^2 + y_2^2 = \frac{1}{2} .
\end{equation}
From the above equations and \eqref{eq10} one gets $x_3 = \pm \sqrt{\frac{2}{5}}$. Then, using \eqref{eq11} and \eqref{eq13} one can see that the set of unit points in $\Lambda^*_{\mg} \cap \N$ is $S^1\times S^1$. Hence, $\Lambda^*_{\mg} \cap \N = (S^1 \times S^1) \times \R$, showing that $\Lambda^*_{\mg}$ and $\N$ are transversal.

Now, let us write the Euler equation of $\mg$. If $u(t) = \sum_{i=1}^3 \alpha_i(t) e_i + \sum_{j=1}^3 \beta_j(t) \i e_j$ is a solution of $u'(t) = [u(t) , A_{\mg}^{-1} u(t)]$, with $u(0)$ an idempotent, then
\begin{align*}
\alpha_1' & = -\sqrt{2} \alpha_2 \alpha_3 + 4\sqrt{2} \beta_2 \beta_3 , & \beta_1' & = 4\sqrt{2} \alpha_3 \beta_2 , \\
\alpha_2' & = \sqrt{2} \alpha_1 \alpha_3 - 4\sqrt{2} \beta_1 \beta_3 , & \beta_2' & = 5\sqrt{2} \alpha_3 \beta_1 , \\
\alpha_3' & = 0 , & \beta_3' & = -4\sqrt{2} \alpha_1 \beta_2 + 4\sqrt{2} \alpha_2 \beta_1 . 
\end{align*}
So, $\alpha_3$ is constant and since it is the coefficient of the timelike vector $e_3$, the rest of the coefficients have to be bounded. Hence, any solution of the Euler equation is complete and, thus, the metric $\mg$ is complete.

Note also that the above equations easily imply that $u'(0) = [u(0) , A_{\mg}^{-1} u(0)]$ only if $u(0) = 0$. So, the Euler field has no idempotent. 

Theorem~\ref{MainTheorem3} shows that the above example is reflecting a general feature for left-invariant Lorentzian metrics on $SL_2(\C)$ stated in Theorem~\ref{MainTheorem3}. 

Here is the proof of Theorem~\ref{MainTheorem3}.

\begin{proof}[\textbf{Proof of Theorem~\ref{MainTheorem3}}]
Suppose that the cones $\Lambda^*_{\mg}$ and $\N$ are transversal at some point $z \in \g{sl}_2(\C)$. Since $\mg$ is Lorentzian, one can see that the induced bilinear form $\mg^*$ is also Lorentzian. The null cone $\Lambda^*_{\mg}$ of the Lorentzian vector space $\g{sl}_2(\C)$ is known to be diffeomorphic to $S^4 \times \R$, see for example \cite{ONe83}. On the other hand, according to \cite{ColMcG93}, $\N-\{0\}$, which is the orbit of  the adjoint action of $G$ containing $z$, is a four-dimensional submanifold of $\g{sl}_2(\C)$ with fundamental group $\mathbb{Z}_2$. Let us fix an arbitrary norm on $\g{sl}_2(\C)$ and denote by $S^5$ the unit sphere in $\g{sl}_2(\C)$ determined by this norm. Then, knowing that $\N$ contains the one-dimensional subspaces generated by each of its elements, it follows that $\N$ and $S^5$ are transversal at all points of their intersection. So, $\N \cap S^5$ is a three-dimensional compact manifold and $\N$ is diffeomorphic to $(\N \cap S^5) \times \R$. In particular, $\N \cap S^5$ has the same fundamental group as $\N$ itself, and thus it is diffeomorphic to the three-dimensional real projective space $\R P^3$. In both products $\Lambda^*_{\mg} = S^4 \times \R$ and $\N = \R P^3 \times \R$ the factors $\R$ represents the one-dimensional linear subspaces generated by elements in $\Lambda^*_{\mg} \cap \N$. So, it follows that $S^4$ and $\R P^3$ are also transversal and, thus, $\Lambda^*_{\mg} \cap \N$ is diffeomorphic to $M \times \R$ for some two-dimensional compact orientable manifold $M \subset S^4 \cap \R P^3$. Indeed, $M$ is the transversal intersection $S^4 \cap \R P^3$. The Euler field $F_{\mg}(x)=[x,A_{\mg}^{-1}x]$ is tangent to both $\Lambda^*_{\mg}$ and $\N$ and, consequently, tangent to their intersection $M \times \R$. The tangent component of $F$ on $M$ defines a vector field on $M$ which we denote by $F_M$. Clearly, if $F_M$ vanishes at some point $x\in M$, then $F_{\mg}(x)$ is radial, that is, $F_{\mg}(x) = ax$ with $0 \neq a \in \R$, which yields $F_{\mg}(\frac{1}{a}x) = [\frac{1}{a}x,A_{\mg}^{-1}(\frac{1}{a}x)] = \frac{1}{a}x$. Hence, $\frac{1}{a}x$ is an idempotent of $\g{sl}_2(\C)$. However, we know from Poincare-Hopf theorem that the $2$-torus $T^2$ is the only two-dimensional compact orientable manifold which admits a non-vanishing vector filed. So if there is no idempotent, then $M$ must be a two-dimensional torus.
\end{proof}

A point worth mentioning here is that on $SL_2(\R)$ completeness of a left-invariant metric is fully revealed just by knowing whether $\N$ and $\Lambda_{\mg}^*$ are transversal, tangent or disjoint. However, generally speaking, even completely determining $\N \cap \Lambda_{\mg}^*$ might not be enough to draw conclusion about completeness of the metric. For instance, the example on page~\pageref{example1-sl2c} and the type of examples mentioned in  Remark~\ref{remark2} both satisfy Theorem~\ref{MainTheorem3}, but the first one is complete, while the second is incomplete.

\subsection{A remark on limit curves}\label{limit-curve}
Our aim in this subsection is to construct a counterexample for a claim in \cite{Yur92}. In \cite{Yur92} the statement below was assumed to be true:

\begin{itemize}
\item[(*)] A limit curve of any sequence of incomplete geodesics in a compact Lorentzian manifold is closed.
\end{itemize}

We show that this is not true in general.
More precisely, we give an example of a compact Lorentzian manifold of the form $SL_2(\C)/\Gamma$ and a sequence of incomplete geodesics in $SL_2(\C)/\Gamma$, so that their limit geodesic is not closed.

Suppose that $\Gamma$ is a finitely generated discrete subgroup of $SL_2(\C)$ acting properly discontinuous on $SL_2(\C)$ such that the quotient space $SL_2(\C)/\Gamma$ is a compact manifold. Existence of such discrete subgroups and their actions are studied thoroughly in a general framework in \cite{Ghy95, GueKas17, Tho18}. Let $\mg$ be a left-invariant Lorentzian metric on $SL_2(\C)$ and as in Lemma~\ref{AdG-family}, for each $g\in SL_2(\C)$ let $\mg_g$ be the metric defined by $\mg_g(x,y):=\K(x,\Ad_g \circ A_{\mg} \circ \Ad_{g^{-1}} (y))$. Any left-invariant Lorentzian metric on $SL_2(\C)$ induces a Lorentzian metric on $SL_2(\C)/\Gamma$. With each of the metrics $\mg$ or $\mg_g$ on $SL_2(\C)$ and the induced metric on $SL_2(\C)/\Gamma$, the quotient map $\pi: G \to SL_2(\C)/\Gamma$ is a semi-Riemannian covering. So, a geodesic in $SL_2(\C)$ is complete if and only if its projection in $SL_2(\C)/\Gamma$ is complete.  

Let us use the same notations and basis introduced in Lemma~\ref{lema3-spacelike=sl2c} for \eqref{eq18} and \eqref{eq20}. In that setting, let $\mg$ be the left-invariant Lorentzian metric on $SL_2(\C)$ whose associated isomorphism on $\mg$ is $A_{\mg}^{-1} = \diag(d_1, d_3, a, a, b, b)$ so that $d=(d_1-a)(d_1-b)>0$. Then, as we saw $\mg$ is incomplete. Indeed, one can pick a nilpotent $\theta = \xi + y$ in the hyperspace $\K(x , \i \xi) = 0$ with $y\in [\i\xi , \g{sl}_2(\C)]$ so that $[\theta , A_{\mg}^{-1}\theta] = \theta$. As it was mentioned in Section~\ref{section2}, the integral curve $u(t)$ of the Euler field $F_{\mg}$ with initial condition $u(0) = \theta$ is $u(t) = \alpha(t)\theta$, where $\alpha(t) = 1/(1-t)$. The corresponding geodesic of $u(t)$ in $SL_2(\C)$ is $\gamma_{\theta}(t) = \exp(\beta(t)A_{\mg}^{-1}\theta)$ where $\beta$ satisfies $\beta'(t) = \alpha(t)$ and $\beta(0) = 1$. Clearly, $\gamma_{\theta}(t)$ is incomplete since $u(t)$ is incomplete. Moreover, we have $A_{\mg}^{-1}\theta = \frac{1}{2}\xi' + r \theta$, for some $r\in \R$, where $\xi'$ is the semisimple element of the $\g{sl}_2$-triple containing $\theta$. Using this expression for $A_{\mg}^{-1}\theta $, one can check that $\gamma_{\theta}$ is not closed.

For each positive integer $m$, let $u_m(t)$ be the curve in $\g{sl}_2(\C)$ defined by
\[
u_m(t) = g_m(t)\xi + h_m(t)y ,
\]
where with $a_m = \sqrt{\frac{1}{m}}$ and $b_m = (1-a_m)/(1+a_m)$, the functions $g_m$ and $h_m$ are given by 
\[
g_m(t) = \frac{1+b_m e^{(2a_m)t}}{1-b_m e^{2a_m t}} a_m ,  \quad h_m(t)  = \sqrt{1-\frac{1}{d_1 m}}e^{\int_0^t g_m(s)ds} .
\]

One can check that each $u_m(t)$ is an incomplete solution of the Euler equation of $\mg$. Furthermore, $u_m(0) \to \theta$, that it, $\gamma_{\theta}(t)$ is a limit geodesic of the sequence $\{\gamma_m(t)\}$, where, for each $m$, $\gamma_m(t)$ denotes the reflected geodesic of $u_m(t)$ in $SL_2(\C)$.

Now, we consider the projection of the geodesics $\gamma_{\theta}(t)$ and $\gamma_m(t)$ on the quotient space $SL_2(\C)/\Gamma$. Since the projection map $\pi$ is a semi-Riemannian covering map, one can see that the geodesic $\pi(\gamma_{\theta}(t))$ is a limit geodesic of the sequence of incomplete geodesics $\{\pi(\gamma_m(t))\}$ in $SL_2(\C)/\Gamma$. If the geodesic $\pi(\gamma_{\theta}(t))$ is not closed, then we have the example we were looking for. So, suppose that $\pi(\gamma_{\theta}(t))$ is closed. This means that $\gamma_{\theta}(t) \cap \Gamma$ must include an element $q_e$ other than the identity.

For each $g\in G$, we can repeat the above procedure for the metric $\mg_g$ and the idempotent $\theta_g = \Ad_g\theta$ instead of $\mg$ and $\theta$. Each time we obtain a new element $e \neq q_g \in \Gamma$, which is in $\gamma_{\theta_g}(t) \cap \Gamma$. The set $\{q_g : g \in G\}$ has the same cardinality as $\{\Ad_g \theta : g \in SL_2(\C) \}$. The latter set is, indeed, $\N-\{0\}\subset \g{sl}_2(\C)$ which, as we mentioned before, is a four-dimensional manifold. This contradicts the assumption that $\Gamma$ is finitely generated. So, for some of the metrics $\mg_g$, the geodesic $\pi(\gamma_{\theta_g})$ must not be closed which provides us with the type of an example that we were looking for.

\bibliographystyle{amsplain}

\end{document}